\documentclass[10pt,twocolumn,twoside]{IEEEtran}
\usepackage{amsmath,amssymb} 
\usepackage{amssymb}  
\usepackage{tikz}
\usepackage{multicol}
\usepackage{blindtext}
\def\E{\mathbf{E}}

\usepackage{amsmath}
\usepackage{amssymb}
\usepackage{psfrag,graphicx}

\usepackage{epsfig}
\usepackage{verbatim}
\newtheorem{thm}{Theorem}

\newtheorem{prp}{Proposition}
\newtheorem{ex}{Example}

\def\sym{\mathbb{S}}

\def\R{\mathbb{R}}
\def\E{\mathbf{E}}
\def\K{\mathbb{K}}

\def\QED{~\rule[-1pt]{6pt}{6pt}\par\medskip}
\newenvironment{proof}{{\bf Proof.\ }}{ \hfill \QED}  \title{\LARGE \bf Team Decision Problems with Convex Quadratic Constraints}
\author{Ather Gattami
\thanks{Ather Gattami is with Ericsson Research,
Stockholm, Sweden. E-mail: 
ather.gattami@ericsson.com}
}
\begin{document}
\maketitle
\begin{abstract}
In this paper, we consider linear quadratic team problems with an arbitrary number of quadratic constraints in both stochastic and deterministic settings. The team consists of players with different measurements about the state of nature. The objective of the team is to minimize a quadratic cost subject to additional finite number of quadratic constraints. We first consider the problem of countably infinite number of players in the team for a bounded state of nature with a Gaussian distribution and show that linear decisions are optimal. Then, we consider the problem of team decision problems with additional convex quadratic constraints and show that linear decisions are optimal for both the finite and infinite number of players in the team. For the finite player case, the optimal linear decisions can be found by solving a semidefinite program. Finally, we consider the problem of minimizing a quadratic objective for the worst case scenario, subject to an arbitrary number of deterministic quadratic constraints. We show that linear decisions are optimal and can be found by solving a semidefinite program. Finally, we apply the developed theory on dynamic team decision problems in linear quadratic settings.
\end{abstract}

\section{Introduction}
We consider the problem of distributed decision
making with information constraints and quadratic constraints under linear quadratic
settings. For instance, information constraints appear naturally
when making decisions over networks. Quadratic constraints appear due to the power limited controllers in practice for instance. These problems can be
formulated as team problems. The team problem is an optimization problem
with several decision makers possessing different information aiming to optimize a common
objective. Early results in \cite{marschak:1955}  considered static team
theory in stochastic settings and a more general framework was introduced
by Radner \cite{radner}, where existence and uniqueness of solutions where shown.
Generalization to dynamic team problems for control purposes where introduced in \cite{ho:chu}. 
In \cite{didinsky:basar:1992}, the deterministic team problem
with two team members was solved. The solution can't be easily
extended to more than two players since it uses the fact that the
two members have common information; a property that doesn't
necessarily hold for more than two players. Also, a nonlinear team
problem with two team members was considered in
\cite{bernhard:99}, where one of the team members is assumed to
have full information whereas the other member has only access to
partial information about the state of the world. Related team
problems with exponential cost criterion were considered in
\cite{krainak:82}. Optimizing team problems with respect to
\textit{affine} decisions in a minimax quadratic cost was shown to
be equivalent to stochastic team problems with exponential cost,
see \cite{fan:1994}. The connection is not clear when the
optimization is carried out over nonlinear decision functions.
The deterministic version (minimizing the worst case scenario) 
of the linear quadratic team decision problem was solved in \cite{gattami:bob:rantzer}.
The problem of countably infinite number of players under the power semi-norm was
solved in \cite{mahajan:2013} under certain assumptions.

In this paper, we will consider both Gaussian and deterministic settings(worst case scenario)
for team decision problems under additional quadratic constraints. It's
well-known that additional constraints, although convex, could
give rise to complex optimization problems if the optimized variables are functions
(as opposed to real numbers). For instance, linear functions (that is functions of the 
form $\mu(x) = Kx$ where $K$ is a real matrix) are no longer optimal. We will illustrate this fact by the following example.

\begin{ex}
\label{counterex}
For $x\in \mathbb{R}$, we want to minimize the objective function 
$$|u|^2$$  
subject to $$|x-u|^2\leq \gamma$$

Some Hilbert space theory shows that the optimal $u$ is given by

$$
u=\mu(x) = (|x|-\sqrt{\gamma})x/|x| ~~ \text{if}~~  |x|^2>\gamma, 
$$
and
$$
u=\mu(x) = 0 ~~ \text{otherwise}. 
$$
Obviously, the optimal $u$ is a nonlinear function of $x$. \\
\end{ex}

Increasing the dimension of $x$, and adding constraints on the structure of $u$, for instance $x\in \mathbb{R}^N$ 
and $u=\mu(x)=(\mu(x_1), .., \mu(x_N))$, certainly makes the constrained optimization more complicated.
The example above shows that, in spite of having a convex optimization carried out 
over a Hilbert space, the optimal decision function is nonlinear. However, we show in the upcoming sections that
problems multiple quadratic constraints behave nicely when considering the \textit{expected} values of the objectives in the Gaussian case, in the sense that linear decisions are optimal. We also extend the results to the case of infinite number of players in the team. For the deterministic counterpart which is not an optimization problem over a Hilbert space, we show that linear decisions are optimal and we show how to find the linear optimal decisions by semidefinite programming. Finally, we apply the developed theory on dynamic team decision problems in linear quadratic settings.
\section{Notation}
The following table gives a list of the notation we are going to use throughout the paper:\\

\begin{tabular}{ll}
$\mathbb{R}_{+}$& The set of nonnegative real numbers.\\
$\mathbb{N}$& The set of positive integers.\\
$\mathbb{S}^n_{+}$& The set of $n\times n$ symmetric positive\\
& semidefinite matrices.\\
$\mathbb{S}^n_{++}$& The set of $n\times n$ symmetric positive\\
& definite matrices.\\
$\mathcal{M}$& The set of measurable functions.\\
$\mathcal{C}$& The set of functions
$\mu:\mathbb{R}^p\rightarrow \mathbb{R}^m$ with\\
& $\mu(y)=(\mu_1^*(y_1),\mu_2^*(y_2), ..., \mu_N^*(y_N))^*$,\\
& $\mu_i:\mathbb{R}^{p_i}\rightarrow \mathbb{R}^{m_i}$,
$\sum_{i} m_i=m$, $\sum_{i} p_i=p$.\\
$A_{i}$& The element of $A$ in row position $i$.\\
$\succeq$ & $A\succeq B$ $\Longleftrightarrow$ $A-B\in \mathbb{S}^n_{+}$.\\
$\succ$ & $A\succ B$ $\Longleftrightarrow$ $A-B\in \mathbb{S}^n_{++}$.\\
$\mathbb{K}$ 	& $\mathbb{K} = \{K| K=\mathbf{diag}(K_1, ..., K_N),$ \\
			& \hspace{30mm}$K_i\in \mathbb{R}^{m_i\times p_i}\}$.\\
$\mathbf{Tr}$& $\mathbf{Tr}\{A\}$ is the trace of the matrix $A$.\\
$\mathcal{N}(m,X)$&  The set of Gaussian variables with\\
& mean $m$ and covariance $X$.
\end{tabular}


\section{Linear Quadratic Gaussian Team Theory}
In this section, we will review a classical result in
stochastic team theory for a finite number of decision variables
and present an extension to the case of infinite number of decision variables. 

In the static team decision problem, one would like to
solve
\begin{equation}
\label{static}
\begin{aligned}
\min_{\mu}\hspace{2mm} & \mathbf{E}\left[
\begin{matrix}
x\\
u
\end{matrix}
\right]^*
\left[
\begin{matrix}
Q_{xx} & Q_{xu}\\
Q_{ux} & Q_{uu}
\end{matrix}
\right]
\left[
\begin{matrix}
x\\
u
\end{matrix}
\right]\\
\text{subject to } & y_i=C_ix+v_i\\
                   & u_i = \mu_i(y_i)\\
                   & \text{for } i=1,..., N.
\end{aligned}
\end{equation}
Here, $x$ and $v$ are independent Gaussian variables taking values in
$\mathbb{R}^n$ and $\mathbb{R}^{p}$, respectively, with
$x\sim \mathcal{N}(0,V_{xx})$ and $v\sim \mathcal{N}(0,V_{vv})$.
Also, $y_i$ and $u_i$ will be stochastic variables taking values in
$\mathbb{R}^{p_i}$, $\mathbb{R}^{m_i}$, respectively, and
$p_1+...+p_N=p$. We assume that
\begin{equation}
\left[\begin{matrix}
Q_{xx} & Q_{xu}\\
Q_{ux} & Q_{uu}
\end{matrix}\right]\in \mathbb{S}^{m+n},
\end{equation}
and $Q_{uu}\in \mathbb{S}^{m}_{++}$, $m=m_1+\cdots+m_N$.

If full state  information about $x$ is available to each
\textit{decision maker} $u_i$, the minimizing $u$ can be
found easily by completion of squares. It is given by $u=Lx$, where
$L$ is the solution to
$$Q_{uu}L=-Q_{ux}.$$
Then, the cost function in (\ref{static}) can be rewritten as
\begin{equation}
\label{cost}
\begin{aligned}
J(x,u) & = \mathbf{E}\{x^T(Q_{xx}-L^TQ_{uu}L)x\}+\\
       & \hspace{5mm}\mathbf{E}\{(u-Lx)^TQ_{uu}(u-Lx)\}.
\end{aligned}
\end{equation}
Minimizing the cost function $J(x,u)$, is equivalent to
minimizing $$\mathbf{E}\{(u-Lx)^TQ_{uu}(u-Lx)\},$$
since nothing can be done about
$\mathbf{E}\{x^T(Q_{xx}-L^TQ_{uu}L)x\}$
(the cost when $u$ has full information).

The next result is due to Radner \cite{radner}, showing that linear decision are optimal for the finite-dimensional static team problem

\begin{prp}[Radner]
\label{radner1}
Let $x$ and $v_i$ be Gaussian variables with zero mean,
taking values in $\mathbb{R}^n$ and $\mathbb{R}^{p_i}$, respectively,
with $p_1+...+p_N=p$.
Also, let $u_i$ be a stochastic variable taking values in
$\mathbb{R}^{m_i}$,
$Q_{uu}\in \mathbb{S}^{m}_{++}$, $m=m_1+\cdots+m_N$,
$L\in \mathbb{R}^{m\times n}$, $C_i\in \mathbb{R}^{p_i\times n}$,
for $i=1, ..., N$.
Then, the optimal decision $\mu$ to the optimization problem
\begin{equation}
\label{opt1}
\begin{aligned}
\min_{\mu}\hspace{3mm}   & \mathbf{E}\{(u-Lx)^*Q_{uu}(u-Lx)\}\\
\textup{subject to}\hspace{3mm} & y_i=C_ix+v_i\\
                   & u_i = \mu_i(y_i)\\
                   & \text{for } i=1,..., N.
\end{aligned}
\end{equation}
is unique and linear in $y$.\\
\end{prp}

\begin{proof}
For a proof, see \cite{radner}.
\end{proof}

It's not clear how to extend the result above to the case of infinite number of state and decision
variables, that is $N = \infty$. This is an important case to approach dynamic team problems, where the decision variables are over space and \textit{time}, and
the time horizon that goes to infinity. 

The next theorem establishes a generalization of the above proposition for the infinite-dimensional case. 

\begin{thm}
\label{infradner}
Let $x = (x_1^* ~ x_2^* ~ ...)^* $ and $v = (v_1^*  ~ v_2^* ~ ...)^*$ be infinite-dimenational vectors where
$x_i$ is Gaussian taking values in $\mathbb{R}^{n_i}$ and $v_i$ is Gaussian taking values in $\mathbb{R}^{p_i}$.
Suppose that
$
\E\{x^*x\}, \E\{v^*v\}  < \infty.
$
Also, let $u = (u_1^*~ u_2^*~ ...)^* $ be a random infinite-dimensional vector with $u_i$ taking values in
$\mathbb{R}^{m_i}$, 
$Q_{uu}$ an infinite dimensional, bounded, positive definite, self-adjoint linear operator, 
$L$ a bounded linear operator, and $C_i\in \mathbb{R}^{p_i\times \infty}$ a bounded linear operator, 
for all $ i\in \mathbb{N}$.
Then, the optimal decision $\mu$ to the optimization problem
\begin{equation}
\label{opt1}
\begin{aligned}
\inf_{\mu}\hspace{3mm}   & \mathbf{E}\{(u-Lx)^* Q_{uu}(u-Lx)\}\\
\textup{subject to}\hspace{3mm} & y_i=C_ix+v_i\\
                   & u_i = \mu_i(y_i)\\
                   & \text{for all } i\in \mathbb{N}.
\end{aligned}
\end{equation}
is unique and linear in $y$.\\
\end{thm}

\begin{proof}
Note first that $y_i$ is bounded since $C_i$, $x$, and $v_i$ are bounded.

Let $\mathcal{Z}$ be the linear space of functions such that $z
\in \mathcal{Z}$ if $z_i$ is a linear transformation of $y_i$,
that is $z_i=A_i y_i$ for some real matrix $A_i \in
\mathbb{R}^{m_i\times p_i}$.
Since $Q_{uu}$ is a bounded symmetric positive definite linear operator, $\mathcal{Z}$ is a linear
space under the inner product
$$\langle g,h \rangle =\mathbf{E}\{g^* Q_{uu}h\},$$
and norm
$$\|g\|^2=\mathbf{E}\{g^* Q_{uu}g\}.$$
The optimization problem in (\ref{opt1}) where we search for the \textit{linear}
optimal decision can be written as
\begin{equation}
\label{optinner}
\begin{aligned}
\min_{u\in \mathcal{Z}}\hspace{3mm}   & \|u-Lx\|^2
\end{aligned}
\end{equation}
Finding the best linear optimal decision $u_\star \in \mathcal{Z}$ to the above
problem is equivalent to finding the shortest distance from the
subspace $\mathcal{Z}$ to the element $Lx$ ($Lx$ is bounded since $L$ and $x$ are bounded), where the minimizing
$u_\star$ is the projection of $Lx$ on $\mathcal{Z}$, and hence
unique. Also, since $u_\star$ is the projection, we have
$$0=\langle u_\star-Lx, u \rangle=\mathbf{E}\{(u_\star-Lx)^* Q_{uu}u\},$$
for all $u\in \mathcal{Z}$. In particular, for
$f_i=(0,0,...,z_i,0, 0, ...)\in \mathcal{Z}$, we have
$$
\mathbf{E}\{(u_\star-Lx)^* Q_{uu}f_i\} =
\mathbf{E}\{[(u_\star-Lx)^* Q_{uu}]_i z_i\}=0.
$$
The Gaussian assumption implies that $$[(u_\star-Lx)^* Q_{uu}]_i$$
is independent of $z_i=A_iy_i$, for all linear transformations
$A_i$. This gives in turn that $[(u_\star-Lx)^* Q_{uu}]_i$ is
independent of $y_i$. Hence, for any decision $\mu\in
\mathcal{M}\cap \mathcal{C}$, \textit{linear or nonlinear}, we
have that
\begin{equation*}
\begin{aligned}
\mathbf{E}(u_\star-Lx)^* Q_{uu}\mu(y)&=
\sum_i \mathbf{E}\{[(u_\star-Lx)^* Q_{uu}]_i \mu_i(y_i)\}\\
		&=0,
\end{aligned}
\end{equation*}
and
\begin{equation*}
\begin{aligned}
&\mathbf{E}(\mu(y) -  Lx)^*  Q_{uu}(\mu(y)-Lx) \\ 
&= 	\mathbf{E}(u_\star-Lx+\mu(y)-u_\star)^* Q_{uu}(u_\star-Lx+\mu(y)-u_\star)\\
&=  	\mathbf{E}(u_\star-Lx)^* Q_{uu}(u_\star-Lx)\\
&+
 	\mathbf{E}(\mu(y)-u_\star)^*Q_{uu}(\mu(y)-u_\star) \\
&+	2\mathbf{E}(u_\star-Lx)^*Q_{uu}(\mu(y)-u_\star)\\
&=	\mathbf{E}(u_\star-Lx)^*Q_{uu}(u_\star-Lx)\\
&+
 	\mathbf{E}(\mu(y)-u_\star)^*Q_{uu}(\mu(y)-u_\star)\\
&\geq \mathbf{E}(u_\star-Lx)^*Q_{uu}(u_\star-Lx)
\end{aligned}
\end{equation*}
with equality if and only if $\mu(y)=u_\star$. This
concludes the proof.
\end{proof}

\section{Team Decision Problems with Power Constraints}
In practice, we often have power constraints on the control variables of the from
 $\gamma_i\geq \E |\mu_i(y_i)|^2$. The question is whether linear decisions are optimal
 and if there is a practical algorithm that can obtain optimal decisions. The introductory example 
 clearly showed that linear decisions are not optimal for the case of point-wise optimization with
 a power constraint. Thus, there is no reason to expect that linear decisions are optimal for the stochastic
 (average) case. This will be addressed in this section.

Consider the modified version of the optimization problem (\ref{static}):
\begin{equation}
\label{power}
\begin{aligned}
\min_{\mu}\hspace{2mm} & \mathbf{E}
\left[
\begin{matrix}
x\\
u
\end{matrix}
\right]^*
\left[
\begin{matrix}
Q & S\\
S^* & R
\end{matrix}
\right]
\left[
\begin{matrix}
x\\
u
\end{matrix}
\right]\\
\text{subject to } & y_i=C_ix\\
                   & u_i = \mu_i(y_i)\\
                   & \gamma_i\geq \E |\mu_i(y_i)|^2\\
                   & \text{for } i=1,..., N.
\end{aligned}
\end{equation}
The difference from Radner's original formulation is that we have added 
power constraints to the decision functions, $\gamma_i\geq \E |\mu_i(y_i)|^2$. Note that additional constraints 
in functional optimization could give rise to complex nonlinear optimal solution as was shown in Example \ref{counterex} in the introduction.


In the sequel, we will prove a more general theorem, where we consider power constraints 
on a set of quadratic forms in both the state $x$ and the decision function $\mu$.

\begin{thm}
\label{gattami}
Let $x$ be a Gaussian variable with zero mean and given covariance 
matrix $X$,
taking values in $\mathbb{R}^n$.
Also, let
$
\left[
\begin{matrix}
Q_0 & S_0\\
S_0^* & R_0
\end{matrix}
\right]\in \sym_{+}^{m+n}
$,   $R_0\in \sym_{++}^m$, 
$
\left[
\begin{matrix}
Q_j & S_j\\
S_j^* & R_j
\end{matrix}
\right]\in \sym^{m+n}
$,
and $R_j\in \sym_{+}^m$, 
for $j = 1, ..., M$.
Assume that the optimization problem
\begin{equation}
\label{main}
\begin{aligned}
\min_{\mu\in\mathcal{C}}\hspace{2mm} & \mathbf{E}\left[
\begin{matrix}
x\\
\mu(x)
\end{matrix}
\right]^*
\left[
\begin{matrix}
Q_0 & S_0\\
S_0^* & R_0
\end{matrix}
\right]
\left[
\begin{matrix}
x\\
\mu(x)
\end{matrix}
\right]\\
\textup{subject to } & 
\mathbf{E}\left[
\begin{matrix}
x\\
\mu(x)
\end{matrix}
\right]^*
\left[
\begin{matrix}
Q_j & S_j\\
S_j^* & R_j
\end{matrix}
\right]
\left[
\begin{matrix}
x\\
\mu(x)
\end{matrix}
\right] \leq \gamma_j\\ 
& j=1,..., M
\end{aligned}
\end{equation}
is feasible. Then, linear decisions $\mu$ given by $\mu(x) = K(X) x$, with $K(X)\in \mathbb{K}$, are optimal.\\
\end{thm}
\begin{proof}
Consider the expression
\begin{equation*}
\begin{aligned}
\mathbf{E}&
\left[
\begin{matrix}
x\\
\mu(x)
\end{matrix}
\right]^*
\left[
\begin{matrix}
Q_0 & S_0\\
S_0^* & R_0
\end{matrix}
\right]
\left[
\begin{matrix}
x\\
\mu(x)
\end{matrix}
\right]+\\
&\sum_{j=1}^M 
\lambda_j
\left( \mathbf{E}
\left[
\begin{matrix}
x\\
\mu(x)
\end{matrix}
\right]
\left[
\begin{matrix}
Q_j & S_j\\
S_j^* & R_j
\end{matrix}
\right]
\left[
\begin{matrix}
x\\
\mu(x)
\end{matrix}
\right]
-\gamma_j
\right).
\end{aligned}
\end{equation*}
Take the expectation of a quadratic form with index $j$ to be larger than $\gamma_j$. Then, 
$\lambda_j \rightarrow \infty$ makes the value of the expression above infinite. On the other hand, if  the expectation of a quadratic form with index $j$ is smaller than $\gamma_j$,
then the maximizer $\lambda_j$ is optimal for  $\lambda_j = 0$.

Now let $p_\star$ be the optimal value of the optimization problem (\ref{main}), and consider the objective function
\begin{equation*}
\begin{aligned}
\left[
\begin{matrix}
x\\
u
\end{matrix}
\right]^*
\left[
\begin{matrix}
Q_0 & S_0\\
S_0^* & R_0
\end{matrix}
\right]
\left[
\begin{matrix}
x\\
u
\end{matrix}
\right] & = x^*(Q_0-S_0R_0^{-1}S_0^*)x\\
	&+(u+R_0^{-1}S_0^*x)^*R_0(u+R_0^{-1}S_0^*x).
\end{aligned}
\end{equation*}

We have that $Q_0-S_0R_0^{-1}S_0^*\succeq 0$, since it's the Schur complement of $R_0$ in the positive 
semi-definite matrix
$
\left[
\begin{matrix}
Q_0 & S_0\\
S_0^* & R_0
\end{matrix}
\right].
$
Since $R_0\succ 0$, a necessary condition for the objective function to be zero is that $u=-R_0^{-1}S_0^*x$, and so $u$ must be linear
(In order for $u$ to have the structure given by $\mathcal{C}$, 
$R_0^{-1}S_0^*$ must be in $\K$, to satisfy the information 
constraints).
 
Now assume that $p_\star>0$. We have
{\small
\begin{equation}
\label{pstar}
\begin{aligned}
p_\star &= \min_{\mu\in\mathcal{C}} \max_{\lambda_i\in \R_+}\hspace{1mm}
\mathbf{E}
\left[
\begin{matrix}
x\\
\mu(x)
\end{matrix}
\right]^*
\left[
\begin{matrix}
Q_0 & S_0\\
S_0^* & R_0
\end{matrix}
\right]
\left[
\begin{matrix}
x\\
\mu(x)
\end{matrix}
\right]\\
&+\sum_{j=1}^M 
\lambda_j
\left( \mathbf{E}
\left[
\begin{matrix}
x\\
\mu(x)
\end{matrix}
\right]
\left[
\begin{matrix}
Q_j & S_j\\
S_j^* & R_j
\end{matrix}
\right]
\left[
\begin{matrix}
x\\
\mu(x)
\end{matrix}
\right]
-\gamma_j
\right)\\
&= \min_{\mu\in\mathcal{C}} \max_{\lambda_i\in \R_+}\hspace{1mm} \\
& \mathbf{E}
\left[
\begin{matrix}
x\\
\mu(x)
\end{matrix}
\right]^*
\left(
\left[
\begin{matrix}
Q_0 & S_0\\
S_0^* & R_0
\end{matrix}
\right]
+
\sum_{j=1}^M 
\lambda_i
\left[
\begin{matrix}
Q_j & S_j\\
S_j^* & R_j
\end{matrix}
\right]
\right)
\left[
\begin{matrix}
x\\
\mu(x)
\end{matrix}
\right]\\
&-\sum_{j=1}^M \lambda_j\gamma_j.
\end{aligned}
\end{equation}
}
Now introduce $\lambda_0$ and the matrix
$$
\left[
\begin{matrix}
Q    	& S\\
S^* & R
\end{matrix}
\right]
=
\sum_{j=0}^M 
\lambda_j
\left[
\begin{matrix}
Q_j & S_j\\
S_j^* & R_j
\end{matrix}
\right],
$$
and consider the minimax problem
\begin{equation}
\label{p0}
\begin{aligned}
p_0 &= \min_{\mu\in\mathcal{C}} \hspace{1mm}\max_{\substack{\lambda_j\geq 0\\ \sum_{j=0}^M \lambda_j=1}} \hspace{1mm}  \mathbf{E}
\left[
\begin{matrix}
x\\
\mu(x)
\end{matrix}
\right]^*
\left[
\begin{matrix}
Q & S\\
S^* & R
\end{matrix}
\right]
\left[
\begin{matrix}
x\\
\mu(x)
\end{matrix}
\right]\\
&-\sum_{j=1}^M \lambda_j\gamma_j.
\end{aligned}
\end{equation}
Note that a maximizing $\lambda_0$ must be positive, since $\lambda_0=0$ implies that $p_0\leq 0$,
while $\lambda_0>0$ gives  $p_0> 0$. We can always recover the optimal solutions of (\ref{pstar})
from that of (\ref{p0}) by dividing all variables by $\lambda_0$, that is $p_{\star}=p_0/\lambda_0$,
$ \lambda_j \mapsto \lambda_j/\lambda_0$, and $\mu\mapsto \mu/\lambda_0$.
Now we have the obvious inequality ($\min \max \{\cdot\}\geq \max  \min \{\cdot\}$)

\begin{equation*}
\begin{aligned}
p_0 &\geq 
\max_{\substack{\lambda_j\geq 0\\ \sum_{j=0}^M \lambda_j=1}}\min_{\mu\in\mathcal{C}} \hspace{1mm}  \mathbf{E}
\left[
\begin{matrix}
x\\
\mu(x)
\end{matrix}
\right]^*
\left(
\sum_{j=0}^M 
\lambda_j
\left[
\begin{matrix}
Q_j & S_j\\
S_j^* & R_j
\end{matrix}
\right]
\right)
\left[
\begin{matrix}
x\\
\mu(x)
\end{matrix}
\right]\\
&-\sum_{j=1}^M \lambda_j\gamma_j.
\end{aligned}
\end{equation*}

For any fixed values of $\lambda_j$, we have $R\succ 0$, so Theorem \ref{radner1} gives the
equality
\begin{equation*}
\begin{aligned}
&\min_{\mu\in\mathcal{C}}
\mathbf{E} \left[
\begin{matrix}
x\\
\mu(x)
\end{matrix}
\right]
\left[
\begin{matrix}
Q    	& S\\
S^* & R
\end{matrix}
\right]
\left[
\begin{matrix}
x\\
\mu(x)
\end{matrix}
\right]=\\
&
\min_{K\in\mathbb{K}}
\mathbf{E} \left[
\begin{matrix}
x\\
Kx
\end{matrix}
\right]
\left[
\begin{matrix}
Q    	& S\\
S^* & R
\end{matrix}
\right]
\left[
\begin{matrix}
x\\
Kx
\end{matrix}
\right], 
\end{aligned}
\end{equation*}
where the minimizing $K$ is unique. Thus,

{\small
\begin{equation*}
\begin{aligned}
p_0 & \geq 
\max_{\substack{\lambda_j\geq 0\\ \sum_{j=0}^M \lambda_j=1}}\min_{\mu\in\mathcal{C}} \hspace{1mm}  \mathbf{E}
\left[
\begin{matrix}
x\\
\mu(x)
\end{matrix}
\right]^*
\left(
\sum_{j=0}^M 
\lambda_j
\left[
\begin{matrix}
Q_j & S_j\\
S_j^* & R_j
\end{matrix}
\right]
\right)
\left[
\begin{matrix}
x\\
\mu(x)
\end{matrix}
\right]\\
&-\sum_{j=1}^M \lambda_j\gamma_j\\
&=
\max_{\substack{\lambda_j\geq 0\\ \sum_{j=0}^M \lambda_j=1}}\min_{K\in \K} \hspace{1mm}  \mathbf{E}
\left[
\begin{matrix}
x\\
Kx
\end{matrix}
\right]^*
\left(
\sum_{j=0}^M 
\lambda_j
\left[
\begin{matrix}
Q_j & S_j\\
S_j^* & R_j
\end{matrix}
\right]
\right)
\left[
\begin{matrix}
x\\
Kx
\end{matrix}
\right]\\
&-\sum_{j=1}^M \lambda_j\gamma_j.
\end{aligned}
\end{equation*}
}
The objective function is radially unbounded in $K$ since $R \succ 0$. Hence, it can be restricted to
a compact subset of $\K$. Thus, 

\begin{equation*}
\begin{aligned}
p_0 & \geq 
\max_{\substack{\lambda_j\geq 0\\ \sum_{j=0}^M \lambda_j=1}}\hspace{1mm} \min_{K\in \K}  \hspace{1mm}  \mathbf{E}
\left[
\begin{matrix}
x\\
Kx
\end{matrix}
\right]^*
\left(
\sum_{j=0}^M 
\lambda_j
\left[
\begin{matrix}
Q_j & S_j\\
S_j^* & R_j
\end{matrix}
\right]
\right)
\left[
\begin{matrix}
x\\
Kx
\end{matrix}
\right]\\
&-\sum_{j=1}^M \lambda_j\gamma_j\\
&=
\min_{K\in \K} \max_{\substack{\lambda_j\geq 0\\ \sum_{j=0}^M \lambda_j=1}} \hspace{1mm}  \mathbf{E}
\left[
\begin{matrix}
x\\
Kx
\end{matrix}
\right]^*
\left(
\sum_{j=0}^M 
\lambda_j
\left[
\begin{matrix}
Q_j & S_j\\
S_j^* & R_j
\end{matrix}
\right]
\right)
\left[
\begin{matrix}
x\\
Kx
\end{matrix}
\right]\\
&-\sum_{j=1}^M \lambda_j\gamma_j\\
&\geq \min_{\mu\in\mathcal{C}} \max_{\substack{\lambda_j\geq 0\\ \sum_{j=0}^M \lambda_j=1}}\hspace{1mm}  \mathbf{E}
\left[
\begin{matrix}
x\\
\mu(x)
\end{matrix}
\right]^*
\left(
\sum_{j=0}^M 
\lambda_i
\left[
\begin{matrix}
Q_j & S_j\\
S_j^* & R_j
\end{matrix}
\right]
\right)
\left[
\begin{matrix}
x\\
\mu(x)
\end{matrix}
\right]\\
&-\sum_{j=1}^M \lambda_j\gamma_j\\
&= p_0,
\end{aligned}
\end{equation*}
where the equality is obtained by applying Proposition \ref{minmaxtheorem} in the Appendix, 
the second inequality follows from the fact that the set of linear decisions $Kx$, $K\in \K$, is a subset of $\mathcal{C}$, and the second equality follows from the definition of $p_0$. Hence, linear decisions are optimal, and the proof is complete.
\end{proof}

{\bf Remark:} Although Theorem \ref{infradner} is stated and proved for $y=x$ and
$u=\mu(y)=\mu(x)$, it extends easily to the case $y=Cx$ for any matrix $C$, which is often the case in applications. 
Note also that we may set $N = \infty$ and the result would still hold by using Theorem \ref{gattami} in the proof.

\section{Computation of The Optimal Team Decisions}
The optimization problem that we would like to solve when assuming linear decisions is

\begin{equation}
\label{linopt}
	\begin{aligned}
		\min_{\gamma_0, K \in \K}    &\hspace{2mm} \gamma_0\\
		\text{subject to } &\hspace{2mm}
		\mathbf{E}\left[
		\begin{matrix}
		x\\
		KCx
		\end{matrix}
		\right]^*
		\left[
		\begin{matrix}
		Q_j & S_j\\
		S_j^* & R_j
		\end{matrix}
		\right]
		\left[
		\begin{matrix}
		x\\
		KCx
		\end{matrix}
		\right] \leq \gamma_j, \hspace{2mm}\\&  \hspace{2mm} j=0, ..., M,\\
		&\hspace{2mm} x\sim \mathcal{N}(0,H^2).
	\end{aligned}
\end{equation}
 with $H\succeq 0$.
Note that we can write the constraints as
\begin{equation}
\begin{aligned}
\mathbf{E}\left[
		\begin{matrix}
		x\\
		KCx
		\end{matrix}
		\right]^* &
		\left[
		\begin{matrix}
		Q_j & S_j\\
		S_j^* & R_j
		\end{matrix}
		\right]
		\left[
		\begin{matrix}
		x\\
		KCx
		\end{matrix}
		\right] \\ &=
		\mathbf{E} \left\{\mathbf{Tr} \left[
		\begin{matrix}
		I\\
		KC
		\end{matrix}
		\right]^*
		\left[
		\begin{matrix}
		Q_j & S_j\\
		S_j^* & R_j
		\end{matrix}
		\right]
		\left[
		\begin{matrix}
		I\\
		KC
		\end{matrix}
		\right]		
		xx^*\right\}\\
		&=\mathbf{Tr} H
		\left[
		\begin{matrix}
		I\\
		KC
		\end{matrix}
		\right]^*
		\left[
		\begin{matrix}
		Q_j & S_j\\
		S_j^* & R_j
		\end{matrix}
		\right]
		\left[
		\begin{matrix}
		I\\
		KC
		\end{matrix}
		\right] H,
\end{aligned}
\end{equation}
where we used that $\mathbf{E} xx^* = X = H^2$.
Hence, we obtain a set of convex quadratic inequalities (convex since $R_j\succeq 0$ for all $j$)
$$
\mathbf{Tr} H
		\left[
		\begin{matrix}
		I\\
		KC
		\end{matrix}
		\right]^*
		\left[
		\begin{matrix}
		Q_j & S_j\\
		S_j^* & R_j
		\end{matrix}
		\right]
		\left[
		\begin{matrix}
		I\\
		KC
		\end{matrix}
		\right] H\leq \gamma_j.
$$
There are many existing computational methods to solve convex quadratic optimization problems (see \cite{boyd:vandenberghe:2004}).

Alternatively, we can formulate the optimization problem as a set of linear matrix inequalities as follows. For simplicity, we will assume that $R_j\succ 0$ for all $j$ (The case $R_j\succeq 0$ is analogue with some technical conditions). 

\begin{thm}
The team optimization problem (\ref{linopt}) is equivalent to 
the semi-definite program
\begin{equation}
\label{sd}
	\begin{aligned}
		\min_{\gamma_0, K \in \K}    	&\hspace{2mm} \gamma_0 \\
		\textup{subject to } 	&\hspace{2mm} \mathbf{Tr} P_j \leq \gamma_j 
		\end{aligned}
	\end{equation}
	{\small
\begin{equation*}
	\begin{aligned}
	 0 &\preceq 
 	\left[
	\begin{matrix}
		P_j -HQ_jH - HS_jKCH - HC^*K^*S_j^*H & HC^*K^* R_j\\
		R_jKCH							& R_j	
	\end{matrix}
	\right]\\ 
	& \hspace{3cm} j=0, ..., M.
	\end{aligned}
\end{equation*}}
\end{thm}
\begin{proof}
Introduce the matrices $P_j\in \sym^n$, and write the given constraints as
$$
\gamma_j \geq \mathbf{Tr} P_j
$$
\begin{equation}
	\begin{aligned} 
		P_j-
		H
		\left[
		\begin{matrix}
		I\\
		KC
		\end{matrix}
		\right]^*
		\left[
		\begin{matrix}
		Q_j & S_j\\
		S_j^* & R_j
		\end{matrix}
		\right]
		\left[
		\begin{matrix}
		I\\
		KC
		\end{matrix}
		\right] H \succeq 0.		
	\end{aligned}
\end{equation}
Now we have that
\begin{equation}
	\begin{aligned} 
		0 &\preceq  P_j - 
		H
		\left[
		\begin{matrix}
		I\\
		KC
		\end{matrix}
		\right]^*
		\left[
		\begin{matrix}
		Q_j & S_j\\
		S_j^* & R_j
		\end{matrix}
		\right]
		\left[
		\begin{matrix}
		I\\
		KC
		\end{matrix}
		\right] H\\ 
		&= P_j -HQ_jH - HS_jKCH - HC^*K^*S_j^* H \\
		&- HC^*K^* R_j KCH.		
	\end{aligned}
\end{equation}
Since $R_j\succ 0$, the quadratic inequality above can be transformed to a linear matrix inequality using the Schur complement (\cite{boyd:vandenberghe:2004}), which is given by
{\small
$$
0 \preceq
\left[
	\begin{matrix}
		P_j -HQ_jH - HS_jKCH - HC^*K^*S_j^* H & HC^*K^* R_j\\
		R_jKCH							& R_j	
	\end{matrix}
\right]
$$}
Hence, our optimization problem to be solved is given by
\begin{equation*}
	\begin{aligned}
		\min_{\gamma_0, K \in \K}    	&\hspace{2mm} \gamma_0 \\
		\textup{subject to } 	&\hspace{2mm} \mathbf{Tr} P_j \leq \gamma_j 
		\end{aligned}
	\end{equation*}
	{\small
\begin{equation*}
	\begin{aligned}
	 0 &\preceq 
 	\left[
	\begin{matrix}
		P_j -HQ_j H - HS_jKCH - HC^*K^*S_j^* H & HC^*K^* R_j\\
		R_jKCH							& R_j	
	\end{matrix}
	\right]\\ 
	& \hspace{3cm} j=0, ..., M.
	\end{aligned}
\end{equation*}}
which proves our theorem.
\end{proof}


\section{Deterministic Team Problems with Quadratic Constraints}

\label{minimaxteam}

We considered the problem of static stochastic
team decision in the previous sections. This section
treats an analogous version for the deterministic (or worst case)
problem. For the dynamic setting with partially nested information(which will be discussed in the next section), 
this corresponds to the $\mathcal{H}_{\infty}$ control problem. 
Although the problem formulation is very similar, the ideas of the solution are
considerably different, and in a sense more difficult.

The deterministic problem considered is a quadratic game
between a team of players and nature. Each player has limited
information
that could be different from the other players in the team.
This game is formulated as a minimax problem, where the team is
the minimizer and nature is the maximizer.

Consider the following team decision problem
\begin{equation}
\label{minimax}
\begin{aligned}
\inf_{\mu} \hspace{1mm} & J(x,u)\\
\text{subject to }\hspace{1mm} & y_i = C_ix\\
                   & u_i = \mu_i(y_i)\\
                   & \text{for } i=1,..., N
\end{aligned}
\end{equation}
where $u_i\in \mathbb{R}^{m_i}$, $m=m_1+\cdots + m_N$,
$C_i\in \mathbb{R}^{p_i\times n}$.\\
 $J(x,u)$ is a quadratic cost given by
$$J(x,u)= \sup_{x\leq 1}
 \left[
\begin{matrix}
x\\
u
\end{matrix}
\right] ^*
\left[
\begin{matrix}
Q_{xx} & Q_{xu}\\
Q_{ux} & Q_{uu}
\end{matrix}
\right]
\left[
\begin{matrix}
x\\
u
\end{matrix}
\right] ,
$$
where
$$
\left[
\begin{matrix}
Q_{xx} & Q_{xu}\\
Q_{ux} & Q_{uu}
\end{matrix}
\right] \in \mathbb{S}^{m+n}.
$$
We will be interested in the case $Q_{uu}\succ 0$.
The players $u_1$,..., $u_N$ make up a \textit{team}, which plays against
\textit{nature} represented by the vector $x$, using
$\mu\in \mathcal{S}$, that is
$$
\mu(Cx)=\left[
  \begin{matrix}
    \mu_1(C_1x)\\
    \vdots\\
    \mu_N(C_Nx)
  \end{matrix}
\right] .
$$



Now consider the team problem (\ref{minimax}) and  note that an equivalent condition for
the existence of a decision function $\mu_\star\in \mathcal{C}$ that achieves the value of the game $\gamma_\star$ is that 
$$
\left[
\begin{matrix}
x\\
\mu_\star(Cx)
\end{matrix}
\right] ^* 
\left[
\begin{matrix}
Q & S\\
S^* & R
\end{matrix}
\right]
 \left[
\begin{matrix}
x\\
\mu_\star(Cx)
\end{matrix}
\right] \leq \gamma_\star |x|^2$$
for all $x$. This is equivalent to
$$
\left[
\begin{matrix}
x\\
\mu_\star(Cx)
\end{matrix}
\right] ^*  
\left[
\begin{matrix}
Q-\gamma_\star I & S\\
S^* & R
\end{matrix}
\right]
\left[
\begin{matrix}
x\\
\mu_\star(Cx)
\end{matrix}
\right] \leq 0$$
for all $x$. This is an example of a \textit{quadratic constraint}. We could also have a set of quadratic
constraints that have to be mutually satisfied. For instance, in addition to the minimization
of the worst case quadratic cost, we could have constraints on the induced norms of the
decision functions

$$\frac{|\mu_i(C_ix)|^2}{|x|^2}\leq \gamma_i \hspace{4mm} \text{for all } x\neq 0,  \hspace{4mm} 
i=1, ..., M,$$ 
or equivalently given by the quadratic inequalities 

$$|\mu_i(C_ix)|^2 - \gamma_i |x|^2 \leq 0 \hspace{4mm} \text{for all } x,  \hspace{4mm} 
i=1, ..., M.$$ 
Also, the team members could share a common power source, and the power is proportional to the squared norm of the decisions $\mu_i$:
$$\sum _{i=1}^M |\mu_i(C_ix)|^2 - c |x|^2 \leq 0 \hspace{4mm} \text{for all } x,  $$ 
for some positive real number $c$.

It's not clear whether linear decisions are optimal, since the example given at the introduction indicates
that, in deterministic settings, nonlinear decisions are optimal \textit{pointwise}.

\begin{thm}
Let 
$$\left[
\begin{matrix}
Q_j & S_j\\
S_j^* & R_j
\end{matrix}
\right]\in \sym^{m+n}$$ 
for $j = 0, 1, ..., M$, $R_ 0\in \sym_{++}^m$, 
$R_j\in \sym_{+}^m$ 
for $j = 1, ..., M$.
Suppose that there exists a decision function $\mu\in \mathcal{C}$ such that
\begin{equation}
\label{maindet}
\begin{aligned}
\sup_{x\in \mathbb{R}^n}\hspace{2mm} 
\left[
\begin{matrix}
x\\
\mu(x)
\end{matrix}
\right]^*
\left[
\begin{matrix}
Q_j & S_j\\
S_j^* & R_j
\end{matrix}
\right]
\left[
\begin{matrix}
x\\
\mu(x)
\end{matrix}
\right] \leq 0, \hspace{3mm} j=0,..., M.
\end{aligned}
\end{equation}
Then, there exists a linear decision $\mu(x)= Kx$, $K\in \mathbb{K}$,  
such that (\ref{maindet}) is satisfied.\\
\end{thm}
\begin{proof}
Suppose there exists a decision function $\mu\in \mathcal{C}$ such that
(\ref{maindet}) is satisfied. Then, for any Gaussian variable $x\sim \mathcal{N}(0,X)$, we have that
$$
\E \left[
\begin{matrix}
x\\
\mu(x)
\end{matrix}
\right]^*
\left[
\begin{matrix}
Q_j & S_j\\
S_j^* & R_j
\end{matrix}
\right]
\left[
\begin{matrix}
x\\
\mu(x)
\end{matrix}
\right] \leq 0.
$$
Equivalently, for a given  $x\sim \mathcal{N}(0,X)$, the optimal value $s$ of the optimization problem
\begin{equation}
\label{det}
\begin{aligned}
\min_{\mu\in\mathcal{C}}\hspace{2mm} & \mathbf{E}\left[
\begin{matrix}
x\\
\mu(x)
\end{matrix}
\right]^*
\left[
\begin{matrix}
Q_0 & S_0\\
S_0^* & R_0
\end{matrix}
\right]
\left[
\begin{matrix}
x\\
\mu(x)
\end{matrix}
\right]\\
\textup{subject to } & 
\mathbf{E}\left[
\begin{matrix}
x\\
\mu(x)
\end{matrix}
\right]^*
\left[
\begin{matrix}
Q_j & S_j\\
S_j^* & R_j
\end{matrix}
\right]
\left[
\begin{matrix}
x\\
\mu(x)
\end{matrix}
\right] \leq 0\\ 
& j=1,..., M
\end{aligned}
\end{equation}
must be nonpositive, $s\leq 0$.
But Theorem \ref{gattami} gives that the decision function $\mu(x) = K(X)x$ is optimal, with  $K(X)\in \mathbb{K}$.
Since $\E xx^* = X$, we get the inequalities
\begin{equation}
\label{det2}
\begin{aligned}
0 &\geq \E\left[
\begin{matrix}
x\\
K(X)x
\end{matrix}
\right]^*
\left[
\begin{matrix}
Q_j & S_j\\
S_j^* & R_j
\end{matrix}
\right]
\left[
\begin{matrix}
x\\
K(X)x
\end{matrix}
\right] \\
&= \E \left\{ x^* \left[
\begin{matrix}
I\\
K(X)
\end{matrix}
\right]^*
\left[
\begin{matrix}
Q_j & S_j\\
S_j^* & R_j
\end{matrix}
\right]
\left[
\begin{matrix}
I\\
K(X)
\end{matrix}
\right]x\right\}\\
&= \mathbf{Tr} 
\left[
\begin{matrix}
I\\
K(X)
\end{matrix}
\right]^*
\left[
\begin{matrix}
Q_j & S_j\\
S_j^* & R_j
\end{matrix}
\right]
\left[
\begin{matrix}
I\\
K(X)
\end{matrix}
\right]X   
\end{aligned}
\end{equation}
for all $j$. Now let $\lambda_i\geq0$, $i=0, ..., M$, and
$$
\left[
\begin{matrix}
Q    	& S\\
S^* & R
\end{matrix}
\right]
=
\sum_{j=0}^M 
\lambda_j
\left[
\begin{matrix}
Q_j & S_j\\
S_j^* & R_j
\end{matrix}
\right].
$$
Introduce the set
$$\mathbb{X}=\{X: X\succeq 0, \mathbf{Tr} X=1\}.$$
The fact that for every covariance matrix $X$ there is
a matrix $K(X)$ such that (\ref{det2}) holds implies
$$
 \max_{\lambda_i \geq 0,  X\in \mathbb{X}} ~ \min_{K\in \mathbb{K}}  \mathbf{Tr} 
\left[
\begin{matrix}
I\\
K
\end{matrix}
\right]^*
\left[
\begin{matrix}
Q    	& S\\
S^* & R
\end{matrix}
\right]
\left[
\begin{matrix}
I\\
K
\end{matrix}
\right]X
 \leq 0.
$$
For every fixed $X$, we have a max-min problem which is  convex in $K$ and linear in 
$\lambda_i$, so we can switch the order of the minimization and maximization to get the max-min-max inequality
$$
 \max_{X\in \mathbb{X}} \min_{K\in \mathbb{K}}  \max_{\lambda_i \geq 0} 
 \mathbf{Tr} 
\left[
\begin{matrix}
I\\
K
\end{matrix}
\right]^*
\left[
\begin{matrix}
Q    	& S\\
S^* & R
\end{matrix}
\right]
\left[
\begin{matrix}
I\\
K
\end{matrix}
\right]X
 \leq 0.
$$
 which is equivalent to the existence of a matrix $K\in \mathbb{K}$ such that
$$ 
\left[
\begin{matrix}
I\\
K
\end{matrix}
\right]^*
\left[
\begin{matrix}
Q & S\\
S^* & R
\end{matrix}
\right]
\left[
\begin{matrix}
I\\
K
\end{matrix}
\right]  \preceq 0
$$
for every set of $\lambda_i\geq 0$, $i=0, ..., M$. 
This implies that there must exist a matrix $K\in \mathbb{K}$ such that
\begin{equation}
\label{alphaeq2}
\left[
\begin{matrix}
I\\
K
\end{matrix}
\right]^*
\left[
\begin{matrix}
Q_j & S_j\\
S_j^* & R_j
\end{matrix}
\right]
\left[
\begin{matrix}
I\\
K
\end{matrix}
\right]  \preceq 0,
\end{equation}
for all $j$. Finally, (\ref{alphaeq2}) implies that 
\begin{equation}
\label{alphaeq}
x^*\left[
\begin{matrix}
I\\
K
\end{matrix}
\right]^*
\left[
\begin{matrix}
Q_j & S_j\\
S_j^* & R_j
\end{matrix}
\right]
\left[
\begin{matrix}
I\\
K
\end{matrix}
\right]x  \leq 0 ~~~\textup{for all } ~ x, ~~~ j= 0, ..., M,
\end{equation}
and the proof is complete.
\end{proof}

%
%

\begin{thm}
\label{gattami2}
Let $C_i\in \mathbb{R}^{p_i\times n}$,
for $i=1, ..., N$. Let
$
\left[
\begin{matrix}
Q_j & S_j\\
S_j^* & R_j
\end{matrix}
\right]\in \sym^{m+n}
$ for $j=0, ..., M$,  
and $R_j\in \sym_{+}^m$  for $0 = 1, ..., M$.
Then, the set of quadratic matrix inequalities 
\begin{equation}
\begin{aligned}
\left[
\begin{matrix}
x\\
KCx
\end{matrix}
\right]^*
\left[
\begin{matrix}
Q_j & S_j\\
S_j^* & R_j
\end{matrix}
\right]
\left[
\begin{matrix}
x\\
KCx
\end{matrix}
\right] \leq 0  ~~~~\forall x, \hspace{3mm} j=0,..., M,
\end{aligned}
\end{equation}
is equivalent to 
\begin{equation}
\begin{aligned}
\left[
\begin{matrix}
Q_j + S_j KC + C^*K^*S_j^* & C^*K^* R_j\\
R_jKC & -R_j
\end{matrix}
\right]
\preceq 0, \hspace{3mm} i=0, ..., M.
\end{aligned}
\end{equation}

\end{thm}
\begin{proof}
We have the following chain of inequalities:
$$
\left[
\begin{matrix}
x\\
KCx
\end{matrix}
\right]^*
\left[
\begin{matrix}
Q_j & S_j\\
S_j^* & R_j
\end{matrix}
\right]
\left[
\begin{matrix}
x\\
KCx
\end{matrix}
\right] \leq 0
$$
$$
\Updownarrow
$$
$$
\left[
\begin{matrix}
I\\
KC
\end{matrix}
\right]^*
\left[
\begin{matrix}
Q_j & S_j\\
S_j^* & R_j
\end{matrix}
\right]
\left[
\begin{matrix}
I\\
KC
\end{matrix}
\right] \preceq 0
$$

$$
\Updownarrow
$$

$$
Q_j + S_j KC + C^*K^*S_j^*+ C^*K^*R_jKC \preceq 0
$$

$$
\Updownarrow
$$

$$
A=\left[
\begin{matrix}
Q_j + S_j KC + C^*K^*S_j^* & C^*K^* R_j\\
R_jKC & -R_j
\end{matrix}
\right]
\preceq 0,
$$ 

\noindent where the last equivalence follows from taking the Schur complement of 
$R_j$ in $A$ (see \cite{boyd:vandenberghe:2004}). Hence, our optimization problem becomes
\begin{equation}
\begin{aligned}
\left[
\begin{matrix}
Q_j + S_j KC + C^*K^*S_j^* & C^*K^* R_j\\
R_jKC & -R_j
\end{matrix}
\right]
\preceq 0,
\end{aligned}
\end{equation}
for $i=0, ..., M$. This completes the proof.
\end{proof}


\section{Distributed Linear Quadratic Control with Quadratic Constraints}
In this section, we will treat the distributed
linear quadratic  $\mathcal{H}_{2}$ control problem with information
constraints, which can be seen as a dynamic team decision
problem. The idea is to transform the dynamic team problem to a 
static one, and then explore information structures for every time step.
The information structure we will be concerned with is the \textit{partially nested}
information structure which was introduced in \cite{ho:chu}.

\begin{figure}
\label{graph}
	\center
	\includegraphics[width = 1\columnwidth]{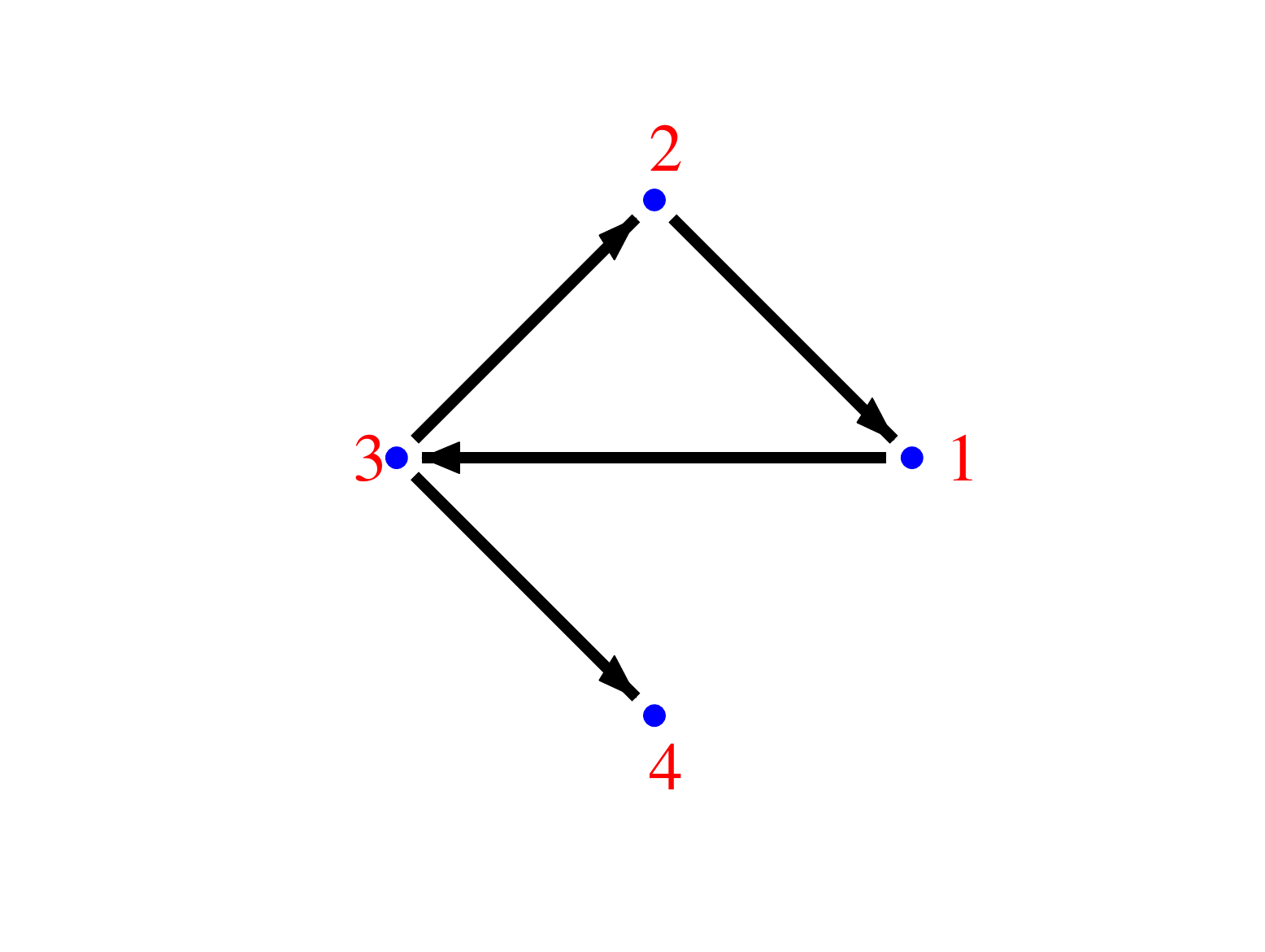}
	\caption{The graph reflects the interconnection structure of the
  dynamics between four systems. The arrow from node 2 to node 1
  indicates that system 1 affects the dynamics of system 2 directly.}
\end{figure} 

Consider an example of four dynamically coupled systems
according to the graph in Figure \ref{graph}. The
equations for the interconnected system are given by  
\begin{equation}
\label{controlsys}
\begin{aligned}
\underbrace{
\left[\begin{matrix}
x_1(k+1)\\
x_2(k+1)\\
x_3(k+1)\\
x_4(k+1)
\end{matrix}\right]}_{x(k+1)} &= 
\underbrace{\left[\begin{matrix}
A_{11} & 0      & A_{13} & 0\\
A_{21} & A_{22} & 0      & 0\\
0      & A_{32} & A_{33} & A_{34}\\
0      & 0      & 0      & A_{44}
\end{matrix}\right]}_{A}
\underbrace{\left[\begin{matrix}
x_1(k)\\
x_2(k)\\
x_3(k)\\
x_4(k)
\end{matrix}\right]}_{x(k)}\\ &+
\underbrace{\left[\begin{matrix}
B_{1}  & 0     & 0      & 0\\
0      & B_{2} & 0      & 0\\
0      & 0     & B_{3}  & 0\\
0      & 0     & 0      & B_{4}
\end{matrix}\right]}_{B}
\underbrace{\left[\begin{matrix}
u_1(k)\\
u_2(k)\\
u_3(k)\\
u_4(k)
\end{matrix}\right]}_{u(k)}+
\underbrace{\left[\begin{matrix}
w_1(k)\\
w_2(k)\\
w_3(k)\\
w_4(k)
\end{matrix}\right]}_{w(k)}.
\end{aligned}
\end{equation}

For instance, the arrow from node 2 to node 1 in the graph
means that the dynamics of system 2 are directly affected by system 1,
which is reflected in the system matrix $A$, where the 
block $A_{21}\neq 0$. On the other hand, 
system 2 does not affect system 1 directly,
which implies that $A_{12}=0$. Because of the ``physical'' distance
between the subsystems, there will be some constraints on the
information available to each node.

The observation of system $i$ at time $k$ is given by
$$y_i(k)=C_ix_i(k),$$ 
where
\begin{equation}
\label{C_structure} 
C_i=\left[\begin{matrix}
C_{i1} & 0      & 0      & 0\\
0      & C_{i2} & 0      & 0\\
0      & 0      & C_{i3} & 0\\
0      & 0      & 0      & C_{i4}
\end{matrix}\right].
\end{equation}
Here, $C_{ij}=0$ if system $i$ does not have access to $y_j(k)$. 
The subsystems could exchange information
about their outputs. Every subsystem recieves the information with
some time delay, that is reflected by the interconnection structure.
Let $\mathbb{I}_i^k$ denote the set of observations
$y_j(n)$ and control signals $u_j(n)$ 
available to node $i$ up to time $k$, $n\leq k$,
$j=1,...,N$.

Consider the following (general) dynamic team decision problem with additional quadratic constraints:  
\begin{equation}
\label{lqg}
\begin{aligned}
\inf_{\mu} \hspace{1mm} & J(u,w)\\
\text{subject to }\hspace{1mm} & x(k+1) = Ax(k)+Bu(k)+w(k)\\
                   & y_i(k) = C_ix(k)\\
                   & u_i(k) = \mu_i:\mathbb{I}_i^k\rightarrow \mathbb{R}^{p_i}\\
                   & \text{for } i=1,..., N\\
                   & \sum_{k=0}^{T-1}
\E
\left[
\begin{matrix}
x(k)\\
u(k)
\end{matrix}
\right]^*
Q_j
\left[
\begin{matrix}
x(k)\\
u(k)
\end{matrix}
\right] \leq \gamma_j\\
& \text{for } j=1,..., M
\end{aligned}
\end{equation}
where 
\begin{equation}
\begin{aligned}
J(u,w) &=\E  x^*(T)Q_fx^*(T)\\
&+\sum_{k=0}^{T-1}
\E
\left[
\begin{matrix}
x(k)\\
u(k)
\end{matrix}
\right]^*
Q
\left[
\begin{matrix}
x(k)\\
u(k)
\end{matrix}
\right] \\ 
\end{aligned}
\end{equation}
$Q_f, Q, Q_j\succ 0$,
$x(k)\in \mathbb{R}^n$, $y_i(k)\in \mathbb{R}^{m_i}$, 
$u_i(k)\in \mathbb{R}^{p_i}$.

Now write $x(k)$ and $y(k)$ as
\begin{equation}
\label{expansion}
\begin{aligned}
x(k) &= A^{t}x(k-t)+\sum_{n=0}^{t-1}A^{n}Bu(k-n-1)+\\
&\hspace{21mm}+\sum_{n=0}^{t-1}A^{n}w(k-n-1),\\
y_i(k) &= C_iA^{t}x(k-t)+\sum_{n=0}^{t-1}C_iA^{n}Bu(k-n-1)+\\
&\hspace{25mm}+\sum_{n=0}^{t-1}C_iA^{n}w(k-n-1).
\end{aligned}
\end{equation}
Note that the summation over $n$ is defined to be zero when $t=0$.

The next result is an extension of \cite{ho:chu} for the case of optimal control with additional quadratic constraints  where it presents a condition on the information structure for which a dynamic problem can be transformed to a static team problem. The condition is known as the partially nested information structure.

\begin{thm}
\label{general_struct}
Consider the optimization problem given by (\ref{lqg}) with the exogenous input $w=(w^*(0)~ w^*(1)~...)^*$ such that $\E\{w^*w\}<\infty$.
The problem is equivalent to a static team problem in the form (\ref{main}) if 
\begin{equation}
\label{t1}
\begin{aligned}
&y_j(k)\in \mathbb{I}_i^k\hspace{1mm} \Rightarrow\\ 
&u_j(k-n-1)\in \mathbb{I}_i^k \hspace{1mm }\text{for } [C_iA^{n}B]_{ij}\neq 0 
\end{aligned}
\end{equation}
for all $n$ such that $0 \leq n< t$, $t=0, 1, 2, ...$.
In particular, the optimal solution to the optimization problem 
given by (\ref{lqg}) is linear in the observations $\mathbb{I}_i^k$ if 
condition (\ref{t1}) is satisfied. 
\end{thm}
\begin{proof}
Introduce 
$$
\bar{x}=
\left[\begin{matrix}
w(0)\\
w(1)\\
\vdots\\
\end{matrix}\right],\hspace{4mm}
\bar{u}_i=\left[\begin{matrix}
u_i(0)\\
u_i(1)\\
\vdots\\
\end{matrix}\right],\\
$$
Then, we can write the cost function $J(x,u)$ as
$$
\E
\left[
\begin{matrix}
\bar{x}\\
\bar{u}
\end{matrix}
\right]^*
\bar{Q}
\left[
\begin{matrix}
\bar{x}\\
\bar{u}
\end{matrix}
\right].
$$
for some symmetric positive definite matrix $\bar{Q}$. Similarly, the quadratic constraints
$$
\sum_{k=0}^{T-1}
\E
\left[
\begin{matrix}
x(k)\\
u(k)
\end{matrix}
\right]^*
Q_j
\left[
\begin{matrix}
x(k)\\
u(k)
\end{matrix}
\right] \leq \gamma_j
$$
may be written as 
$$
\E
\left[
\begin{matrix}
\bar{x}\\
\bar{u}
\end{matrix}
\right]^*
\bar{Q}_j
\left[
\begin{matrix}
\bar{x}\\
\bar{u}
\end{matrix}
\right]\leq \gamma_j.
$$
for some symmetric positive definite matrices $\bar{Q}_j$, $j=1, ..., M$.
Consider the expansion given by (\ref{expansion}). 
The problem here is that $y_i(k)$ depends on previous values of 
the control signals $u(n)$ for $n=0,..., k-1$. The components
$u_j(k-n-1)$ that $y_i(k)$ depends on are completely determined by
the structure of the matrix $[C_iA^{n}B]_{ij}$.

Now if condition (\ref{t1}) is satisfied, node $i$  
has the information of
$u_j(k-n-1)$ available at time $k$ if the element 
$[C_iA^{n}B]_{ij}\neq 0$.
Then, node $i$ can form the new output measurement
\begin{equation}
\begin{aligned}
\check{y}_i(k) &= y_i(k)-\sum_{n=0}^{k-1}C_iA^{n}Bu(k-n-1)\\
               &= A^{k}x(0)+\sum_{n=0}^{k-1}C_i A^{n}w(k-n-1).
\end{aligned}
\end{equation}  
Let
$$
\bar{y}_i(k)=
\left [\begin{matrix}
\check{y}_i(0)\\
\check{y}_i(1)\\
\vdots
\end{matrix}\right].
$$
With these new variables introduced, the optimization problem
given by equation (\ref{lqg}) reduces to the following static 
team decision problem:
\begin{equation}
\label{reducedstatic} 
\begin{aligned}
\inf_{\mu}\hspace{1mm} & 
\E
\left[\begin{matrix}
\bar{x}\\
\bar{u}
\end{matrix}\right]^*
\bar{Q}\left[\begin{matrix}
\bar{x}\\
\bar{u}
\end{matrix}\right]\\
\text{subject to  } & u_i(k)=\mu_i(\bar{y}_i(k))\\
                    & \text{for } i=1, 2, ... \\
                    & \E
\left[
\begin{matrix}
\bar{x}\\
\bar{u}
\end{matrix}
\right]^*
\bar{Q}_j
\left[
\begin{matrix}
\bar{x}\\
\bar{u}
\end{matrix}
\right]\leq \gamma_j\\
& \text{for } j=1, 2, ..., M
\end{aligned}
\end{equation}
and the optimal solution $\bar{u}$ is 
linear according to Theorem \ref{gattami}. This completes the proof.
\end{proof}


\section{Conclusions}
We have studied multi-objective linear quadratic optimization of team decisions
in both stochastic and deterministic settings. Constrained decision problems tend
to have nonlinear optimal solutions. We have shown that for the Gaussian and worst case scenario settings, respectively, 
linear decisions are in fact optimal, and we can find the respective linear optimal solutions by
solving a semidefinite program. We also showed 
that linear decision are optimal when the number of players in the time is infinite.
Future work will consider
an $\mathcal{S}$-procedure sort of a result, where we want to find decision function $\mu$ 
such that the inequality $J_0(\mu(x),x)\leq 0$ is satisfied if  $J_1(\mu(x),x)\leq 0$, where $J_0, J_1$ are some
quadratic forms in $\mu$ and $x$. However, this is a much harder problem since the search for linear 
a function $\mu(x)$ is not a convex problem, and it's not clear if it can be convexified.
\section{Acknowledgements}

This work was in part supported by the Swedish Research Council.

\bibliography{../../ref/mybib}

\appendix

\section{Game Theory}
Let $J=J(u,w)$ be a functional defined on a product vector space 
$\mathbb{U}\times \mathbb{W}$, to be minimized by 
$u\in U\subset \mathbb{U}$ and maximized by $w\in W\subset \mathbb{W}$, where
$U$ and $W$ are the constrained sets. This defines a zero-sum game,
with kernel $J$, in connection with which we can introduce two values,
the \textit{upper value}
$$
\bar{J}:=\inf_{u\in U}\sup_{w\in W} J(u,w),
$$  
and the \textit{lower value}
$$
\underline{J}:=\sup_{w\in W}\inf_{u\in U} J(u,w).
$$  
Obviously, we have the inequality $\bar{J}\geq \underline{J}$.
If  $\bar{J}= \underline{J}=J_\star$, then $J_\star$ is called the 
\textit{value} of the zero-sum game. Furthermore, if there exists a
pair $(u_\star\in U, w_\star\in W)$ such that 
$$
J(u_\star,w_\star)=J_\star,
$$
then the pair $(u_\star,w_\star)$ is called a (pure-strategy) 
\textit{saddle-point solution}. In this case, we say that the game
admits a \textit{saddle-point} (in pure strategies). Such a saddle-point
solution will equivalently satisfy the so-called 
\textit{pair of saddle-point inequalities}:
$$
J(u_\star,w)\leq J(u_\star,w_\star)\leq J(u,w_\star),\hspace{3mm} \forall u\in \mathbb{U},
 \forall w\in \mathbb{W}.
$$

\begin{prp}
\label{minmaxtheorem}
Consider a two-person zero-sum game on convex finite dimensional
action sets $U_1\times U_2$, defined by the continuous kernel 
$J(u_1,u_2)$. Suppose that $J(u_1,u_2)$ is strictly convex in 
$u_1$ and strictly concave in $u_2$. Suppose that either
\begin{itemize}
\item[($i$)] $U_1$ and $U_2$ are closed and bounded, or
\item[($ii$)] $U_i\subseteq \mathbb{R}^{m_i}$, $i=1,2$, and 
$J(u_1,u_2)\rightarrow \infty$ as $\|u_1\|\rightarrow \infty$,
and $J(u_1,u_2)\rightarrow -\infty$ as $\|u_2\|\rightarrow \infty$.
\end{itemize}  
Then, the game admits a unique pure-strategy saddle-point 
equilibrium.
\end{prp}
\begin{proof}
 See \cite{basar:olsder:1999}, pp. 177.
\end{proof}

\end{document}